\documentclass{amsart}
\usepackage{amscd,amssymb,enumerate,amsmath}
\usepackage[notref,notcite]{showkeys}
\usepackage[all]{xy}

\newcommand{\ppq}{\leqslant}
\newcommand{\pgq}{\geqslant}

\newcommand{\im}{\operatorname{Im}\nolimits}
\renewcommand{\ker}{\operatorname{Ker}\nolimits}
\newcommand{\car}{\operatorname{char}\nolimits}
\newcommand{\Hom}{\operatorname{Hom}\nolimits}
\newcommand{\Ext}{\operatorname{Ext}\nolimits}
\newcommand{\HH}{\operatorname{HH}\nolimits}
\newcommand{\opp}{\operatorname{op}\nolimits}

\renewcommand{\sp}{\operatorname{sp}\nolimits}

\newcommand{\ot}{\otimes}
\renewcommand{\lq}{\Lambda_{\mathbf q}}
\newcommand{\iq}{I_{\mathbf q}}
\renewcommand{\L}{\Lambda}
\newcommand{\G}{\Gamma}
\newcommand{\oa}{\bar{a}}
\newcommand{\mo}{\mathfrak{o}}
\newcommand{\mt}{\mathfrak{t}}
\newcommand{\rrad}{\mathfrak{r}}

\newcommand{\N}{\mathcal{N}}

\newcommand{\Z}{\mathbb{Z}}

\newcommand{\A}{\mathcal{A}}

\newcommand{\lift}{\mathcal{L}}

\newtheorem{theorem}{{Theorem}}[section]
\newenvironment{thm}{\begin{theorem}}{\end{theorem}}
\newcommand{\bt}{\begin{thm}}
\newcommand{\et}{\end{thm}}

\newtheorem{corollaire}[theorem]{{Corollary}}
\newenvironment{cor}{\begin{corollaire}}{\end{corollaire}}
\newcommand{\bc}{\begin{cor}}
\newcommand{\ec}{\end{cor}}

\newtheorem{lemme}[theorem]{{Lemma}}
\newenvironment{lemma}{\begin{lemme}}{\end{lemme}}
\newcommand{\bl}{\begin{lemma}}
\newcommand{\el}{\end{lemma}}

\newtheorem{proposition}[theorem]{{Proposition}}
\newcommand{\bprop}{\begin{proposition}}
\newcommand{\eprop}{\end{proposition}}

\newtheorem{definition}[theorem]{{Definition}}
\newenvironment{dfn}{\begin{definition} \rm}{\end{definition}}
\newcommand{\bd}{\begin{dfn}}
\newcommand{\ed}{\end{dfn}}
\newcommand{\dd}{\partial}

\newtheorem*{remark}{Remark}
\newcommand{\br}{\begin{remark}}
\newcommand{\er}{\end{remark}}

\begin{document}

\topmargin 0cm
\oddsidemargin 0.5cm
\evensidemargin 0.5cm
\baselineskip=16pt

\title[A family of algebras with finite Hochschild cohomology]{A family of Koszul self-injective algebras with finite Hochschild cohomology}
\author[Parker]{Alison Parker}
\address{Alison Parker\\Department of Pure Mathematics\\
University of Leeds\\
Leeds, LS2 9JT\\
England}
\email{a.e.parker@leeds.ac.uk}
\author[Snashall]{Nicole Snashall}
\address{Nicole Snashall\\Department of Mathematics\\
University of Leicester\\
University Road\\
Leicester, LE1 7RH\\
England}
\email{njs5@mcs.le.ac.uk}



\begin{abstract}
This paper presents an infinite family of Koszul self-injective algebras whose Hochschild cohomology ring is finite-dimensional. Moreover, for each $N \pgq 5$ we give an example where the Hochschild cohomology ring has dimension $N$. This family of algebras includes and generalizes the 4-dimensional Koszul self-injective local algebras of \cite{BGMS}, which were used to give a negative answer to Happel's question, in that they have infinite global dimension but finite-dimensional Hochschild cohomology.
\end{abstract}

\date{\today}
\maketitle

\section*{Introduction}

Let $K$ be a field. Throughout this paper we suppose $m \pgq 1$, and let ${\mathcal Q}$ be the
quiver with $m$ vertices, labelled $0, 1, \ldots, m-1$, and $2m$
arrows as follows:
$$
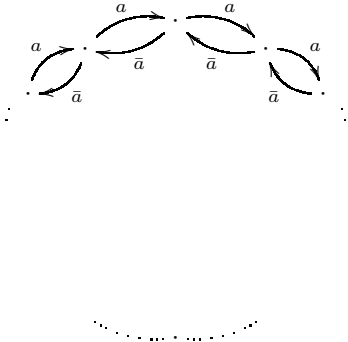
\xymatrix@=.01cm{
&&&&&&&&&&&\cdot\ar@/^.5pc/[rrrrrd]^{a}\ar@/^.5pc/[llllld]^{\bar{a}}\\
&&&&&&\cdot\ar@/^.5pc/[rrrrru]^{a}\ar@/^.5pc/[llldd]^{\bar{a}}&&&&&&&&&&\cdot\ar@/^.5pc/[rrrdd]^{a}\ar@/^.5pc/[lllllu]^{\bar{a}}\\\\
&&&\cdot\ar@/^.5pc/[rrruu]^{a}\ar@{.}@/_.3pc/[ldd]&&&&&&&&&&&&&&&&\cdot\ar@/^.5pc/[llluu]^{\bar{a}}\ar@{.}@/^.3pc/[rdd]\\\\
&&&&&&&&&&&&&&&&&&&&&&&\\
\\
\\
\\\\\\\\\\\\\\\\
&&&&&&&&&&&&&&&&\\
&&&&&&&&&&&\cdot\ar@/_.3pc/@{.}[rrrrru]\ar@/^.3pc/@{.}[lllllu] }$$
Let $a_i$ denote the arrow that goes from vertex $i$ to vertex
$i+1$, and let $\oa_i$ denote the arrow that goes from vertex $i+1$
to vertex $i$, for each $i=0, \ldots , m-1$ (with the obvious
conventions modulo $m$). We denote the trivial path at the vertex
$i$ by $e_i$. Paths are written from left to right.

We define $\Lambda$ to be the algebra $K{\mathcal Q}/I$
where $I$ is the ideal of $K{\mathcal Q}$ generated by
$a_ia_{i+1}$, $\oa_{i-1}\oa_{i-2}$ and $a_i\oa_i-\oa_{i-1}a_{i-1}$,
for $i=0, \ldots , m-1$, where the subscripts are taken modulo $m$.
These algebras are Koszul self-injective special biserial algebras and as
such play an important role in various aspects of the representation
theory of algebras. In particular, for $m$ even, this algebra
occurred in the presentation by quiver and relations of the Drinfeld
double of the generalized Taft algebras studied in \cite{EGST}, and in
the study of the representation theory of $U_q(\mathfrak{sl}_2)$, for which,
see \cite{CK, Patra, Suter, Xiao}.

For $m \pgq 1$ and for each ${\mathbf q} = (q_0, q_1, \ldots ,
q_{m-1}) \in (K^*)^m$, we define $\lq = K{\mathcal Q}/\iq$, where
$\iq$ is the ideal of $K{\mathcal Q}$ generated by $$a_ia_{i+1},
\ \oa_{i-1}\oa_{i-2}, \ q_ia_i\oa_i-\oa_{i-1}a_{i-1} \mbox{ for } i =
0, \ldots , m-1.$$  These algebras are socle deformations of the
algebra $\L$, with $\lq = \L$ when ${\mathbf q} = (1, 1, \ldots , 1)$,
and were studied in \cite{ST2}. We are assuming each $q_i$ is non-zero
since we wish to study self-injective algebras. Indeed, the algebra
$\lq$ is a Koszul self-injective socle deformation of $\L$, and the
$K$-dimension of $\lq$ is $4m$.

In the case $m = 1$, the algebras $\lq$ were studied in \cite{BGMS},
where they were used to answer negatively a question of Happel, in
that their Hochschild cohomology ring is finite-dimensional but they
are of infinite global dimension when $q\in K^*$ is not a root of
unity. In this paper we show, for all $m\pgq 1$, that the algebras
$\lq$, where ${\mathbf q} = (q_0, q_1, \ldots , q_{m-1}) \in (K^*)^m$,
all have finite-dimensional Hochschild cohomology ring when $q_0q_1
\cdots q_{m-1}$ is not a root of unity.
Thus, for each non-zero element of $K$ which is not a root of unity, we have generalized
the 4-dimensional algebra of \cite{BGMS} to an infinite
family of algebras which all give a negative answer to Happel's question. This also complements
the paper of Bergh and Erdmann \cite{BE} in which they extended the example of \cite{BGMS}
by producing a family of local algebras of infinite global dimension for which the
Hochschild cohomology ring is finite-dimensional. We remark that the algebras of
\cite{BE, BGMS} are local algebras with 5-dimensional Hochschild cohomology ring.
In this paper we give, for each $N \pgq 5$, a finite-dimensional algebra with $m=N-4$
simple modules and of infinite global dimension whose Hochschild cohomology ring is $N$-dimensional.

\bigskip

For a finite-dimensional $K$-algebra $\A$ with Jacobson radical
$\rrad$, the Hochschild cohomology ring of $\A$ is given by
$\HH^*(\A) = \Ext^*_{\A^e}(\A,\A) = \oplus_{n \pgq 0}\Ext^n_{\A^e}(\A,
\A)$ with the Yoneda product, where $\A^e = \A^{\opp} \otimes_K \A$ is
the enveloping algebra of $\A$. Since all tensors are over the field
$K$ we write $\otimes$ for $\otimes_K$ throughout. We denote by $\N$
the ideal of $\HH^*(\A)$ which is generated by all homogeneous
nilpotent elements. Thus $\HH^*(\A)/\N$ is a commutative $K$-algebra.

The Hochschild cohomology ring modulo nilpotence of $\lq$, where
${\mathbf q} = (q_0, q_1, \ldots , q_{m-1}) \in (K^*)^m$, was explicitly
determined in \cite{ST2}, where it was shown that $\HH^*(\lq)/\N$ is a
commutative finitely generated $K$-algebra of Krull dimension 2 when
$q_0\cdots q_{m-1}$ is a root of unity, and is $K$ otherwise. Note that, by setting
${\mathbf z} = (q_0q_1\cdots q_{m-1}, 1, \ldots , 1)$, we have an isomorphism $\lq
\cong \L_{\mathbf z}$ induced by $a_i \mapsto q_0q_1\cdots q_ia_i,
\oa_i \mapsto \oa_i$. However, for ease of notation, we will consider
the algebra in the form $\lq = K{\mathcal Q}/\iq$ with ${\mathbf q} =
(q_0, q_1, \ldots , q_{m-1}) \in (K^*)^m$.
It was shown by Erdmann and Solberg in \cite[Proposition 2.1]{ES} that,
if $q_0q_1\cdots q_{m-1}$ is a root of unity, then the finite generation
condition {\bf (Fg)} holds, so that $\HH^*(\lq)$ is a finitely generated
noetherian $K$-algebra. (See \cite{EHSST, ES, SS} for more details on the
finite generation condition {\bf (Fg)} and the rich theory of support
varieties for modules over algebras which satisfy this condition.)

\bigskip

The aim of this paper is to determine $\HH^*(\lq)$ for each $m \pgq 1$ in the case where
$q_0q_1\cdots q_{m-1}$ is not a root of unity, and in particular to
show that this ring is finite-dimensional. Thus we set
$\zeta = q_0q_1\cdots q_{m-1} \in K^*$ and assume
that $\zeta$ is not a root of unity.

\bigskip

\section{The projective resolution of $\lq$}

\bigskip

A minimal projective bimodule resolution for $\Lambda$ was given in
\cite[Theorem 1.2]{ST}. Since $\lq$ is a Koszul algebra, we again use
the approach of \cite{GHMS} and \cite{GSZ} and modify the resolution
for $\Lambda$ from \cite{ST} to give a minimal projective bimodule
resolution $(P^*, \partial^*)$ for $\lq$.

We recall from \cite{H}, that the multiplicity of $\lq e_i\otimes
e_j\lq$ as a direct summand of $P^n$ is
equal to the dimension of $\Ext^n_{\lq}(S_i, S_j)$, where $S_i, S_j$
are the simple right $\lq$-modules corresponding to the vertices $i,
j$ respectively. Thus the projective bimodules $P^n$ are the same as
those in the minimal projective bimodule resolution for $\Lambda$, and
we have, for $n \pgq 0$, that
$$P^n = \oplus_{i=0}^{m-1} [\oplus_{r=0}^n \lq e_i\otimes e_{i+n-2r}\lq].$$

Write $\mo(\alpha)$ for the trivial path corresponding to the origin
of the arrow $\alpha$, so that $\mo(a_i) = e_i$ and $\mo(\oa_i) =
e_{i+1}$. We write $\mt(\alpha)$ for the trivial path corresponding
to the terminus of the arrow $\alpha$, so that $\mt(a_i) = e_{i+1}$
and $\mt(\oa_i) = e_i$. Recall that a non-zero element $r \in
K{\mathcal Q}$ is said to be uniform if there are vertices $v, w$
such that $r = vr = rw$. We then write $v = \mo(r)$ and $w = \mt(r)$.

In \cite{GSZ}, the authors give an explicit inductive construction of
a minimal projective resolution of $\A/\rrad$ as a right $\A$-module, for a
finite-dimensional $K$-algebra $\A$. For
$\A = K{\G}/I$ and finite-dimensional, they define
$g^0$ to be the set of vertices of ${\G}$, $g^1$ to be the
set of arrows of ${\G}$, and $g^2$ to be a minimal set of
uniform relations in the generating set of $I$, and then show that
there are subsets $g^n, n\pgq 3$, of $K{\G}$, where $x\in
g^n$ are uniform elements satisfying $x=\sum_{y\in g^{n-1}}yr_y
=\sum_{z\in g^{n-2}}zs_z$ for unique $r_y,s_z\in K{\G}$,
which can be chosen in such a way that there is a minimal projective
$\A$-resolution of the form
$$\cdots \to Q^4\to Q^3 \rightarrow Q^2 \rightarrow Q^1 \rightarrow Q^0
\rightarrow \A/\rrad \rightarrow 0$$ having the following
properties:
\begin{itemize}
\item[(1)] for each $n\pgq 0$,
$Q^n = \coprod_{x \in g^n} \mathfrak{t}(x)\A$,
\item[(2)] for each $x \in g^n$, there are unique elements $r_j \in
K{\G}$ with $x = \sum_{j}g^{n-1}_jr_j$,
\item[(3)] for each $n\pgq 1$,
using the decomposition of (2), for $x\in g^n$, the map $Q^n
\rightarrow Q^{n-1}$ is given by
$$\mathfrak{t}(x)a \mapsto
\sum_{j}r_j\mathfrak{t}(x)a \mbox{\ \ for all $a \in \A$},$$
\end{itemize}
where the elements of the set $g^n$ are labelled by $g^n =
\{g^n_j\}$. Thus the maps in this minimal projective resolution of
$\A/\rrad$ as a right $\A$-module are described by the
elements $r_j$ which are uniquely determined by (2).

\bigskip

For our algebra $\lq$, we now define sets $g^n$ in the path algebra
$K{\mathcal Q}$ which we will use to label the generators of $P^n$.

\bd
For the algebra $\lq$, $i =0, 1, \ldots , m-1$ and $r = 0, 1, \ldots , n$, define
$$g_{0,i}^0 = e_i$$
and, inductively for $n \pgq 1$,
$$g^n_{r,i} = g^{n-1}_{r,i}a_{i+n-2r-1} +
(-1)^nq_{i-r+1}q_{i-r+2}\cdots q_{i+n-2r}g^{n-1}_{r-1,i}\bar{a}_{i+n-2r}$$
with the conventions that $g^{n-1}_{-1,i} = 0$ and $g^{n-1}_{n,i} = 0$ for all
$n, i$, and that $q_{i-r+1}q_{i-r+2}\cdots q_{i+n-2r} = 1$ if $r=n$.

Define $g^n = \bigcup_{i=0}^{m-1}\{g^n_{r,i} \mid r = 0, \ldots , n\}.$
\ed

It is easy to see, for $n = 1$, that $g_{0,i}^1 = a_i$ and $g_{1,i}^1 =
-\bar{a}_{i-1}$, whilst, for $n=2$, we have $g_{0,i}^2 = a_ia_{i+1}$, $g_{1,i}^2 =
q_ia_i\bar{a}_i - \bar{a}_{i-1}a_{i-1}$ and $g_{2,i}^2 =
-\bar{a}_{i-1}\bar{a}_{i-2}$. Thus
$$\begin{array}{lll}
g^0 & = & \{e_i \mid i = 0, \ldots , m-1\},\\
g^1 & = & \{a_i, -\oa_i \mid i = 0, \ldots , m-1\},\\
g^2 & = & \{a_ia_{i+1},\ \ q_ia_i\oa_i-\oa_{i-1}a_{i-1},\ \ -\oa_{i-1}\oa_{i-2} \mbox{ for all $i$}\},\\
\end{array}$$
so that $g^2$ is a minimal set of uniform relations in the generating set of $\iq$.

Moreover, $g_{r,i}^n \in e_i(K{\mathcal Q})e_{i+n-2r}$, for
$i=0,\ldots, m-1$ and $r=0, \ldots, n$. Since the elements
$g^n_{r,i}$ are uniform elements, we may define
$\mathfrak{o}(g^n_{r,i}) = e_i$ and $\mathfrak{t}(g^n_{r,i}) =
e_{i+n-2r}$. Then $$P^n =
\oplus_{i=0}^{m-1}[\oplus_{r=0}^{n}\lq\mathfrak{o}(g^n_{r,i})
\otimes \mathfrak{t}(g^n_{r,i})\lq].$$

\bigskip

To describe the map $\partial^n\colon P^n\to P^{n-1}$, we need the
following lemma and some notation.

\bl\label{lemma:maps}
For the algebra $\lq$, for $n \pgq 1$, $i =0, 1, \ldots , m-1$ and $r
= 0, 1, \ldots , n$, we have:
$$\begin{array}{ll}
g^n_{r,i}
& = g^{n-1}_{r,i}a_{i+n-2r-1} +
(-1)^n\underbrace{q_{i-r+1}q_{i-r+2}\cdots q_{i+n-2r}}_{\mbox{$n-r$
    terms}} g^{n-1}_{r-1,i}\bar{a}_{i+n-2r} \\
\phantom{a} \\
& = (-1)^r\underbrace{q_{i-r+1}q_{i-r+2}\cdots q_i}_{\mbox{$r$ terms}}
a_ig^{n-1}_{r,i+1} +
(-1)^r\bar{a}_{i-1}g^{n-1}_{r-1,i-1}\end{array}$$
with the
conventions that $g^n_{-1,i} = 0$ and $g^{n-1}_{n,i} = 0$ for all
$n, i$, and that $q_{i-r+1}q_{i-r+2}\cdots q_{i+n-2r} = 1$ if $r=n$ and
$q_{i-r+1}q_{i-r+2}\cdots q_i = 1$ if $r=0$.
Thus
$$g^n_{0,i} = g^{n-1}_{0,i}a_{i+n-1} = a_ig^{n-1}_{0,i+1} \mbox{ and }
g^n_{n,i} = (-1)^ng^{n-1}_{n-1,i}\bar{a}_{i-n} = (-1)^n\bar{a}_{i-1}g^{n-1}_{n-1,i-1}.$$
\el

\begin{proof}
The first formula is of course the definition of $g_{r,i}^n$ so we
need to prove the second equality.
We prove this by induction on $n$. We note that, with the above
conventions, the second formula is correct for $n = 1$ and $2$.

Suppose the second formula is true for $n$ and $n-1$; we consider
the case with $n+1$ and look at the difference:
\begin{align*}
& g^{n}_{r,i}a_{i+n-2r} +
(-1)^{n+1}q_{i-r+1}q_{i-r+2}\cdots q_{i+n-2r+1} g^{n}_{r-1,i}\bar{a}_{i+n-2r+1} \\
 &\quad -  (-1)^rq_{i-r+1}q_{i-r+2}\cdots q_ia_ig^{n}_{r,i+1}
- (-1)^r\bar{a}_{i-1}g^{n}_{r-1,i-1}\\
& =
 (-1)^rq_{i-r+1}q_{i-r+2}\cdots q_ia_ig^{n-1}_{r,i+1} a_{i+n-2r}
+ (-1)^r\bar{a}_{i-1}g^{n-1}_{r-1,i-1} a_{i+n-2r} \\
&\quad + (-1)^{n+1}q_{i-r+1}q_{i-r+2}\cdots q_{i+n-2r+1} (-1)^{r-1}q_{i-r+2}q_{i-r+3}\cdots q_ia_ig^{n-1}_{r-1,i+1} \bar{a}_{i+n-2r+1} \\
&\quad + (-1)^{n+1}q_{i-r+1}q_{i-r+2}\cdots q_{i+n-2r+1} (-1)^{r-1}\bar{a}_{i-1}g^{n-1}_{r-2,i-1} \bar{a}_{i+n-2r+1} \\
 &\quad  -  (-1)^rq_{i-r+1}q_{i-r+2}\cdots q_ia_i
g^{n-1}_{r,i+1}a_{i+n-2r} \\
 &\quad  -  (-1)^rq_{i-r+1}q_{i-r+2}\cdots q_ia_i
(-1)^nq_{i-r+2}q_{i-r+3}\cdots
q_{i+n-2r+1}g^{n-1}_{r-1,i+1}\bar{a}_{i+n-2r+1}
\\
&\quad - (-1)^r\bar{a}_{i-1}
g^{n-1}_{r-1,i-1}a_{i+n-2r}
- (-1)^r\bar{a}_{i-1}
(-1)^nq_{i-r+1}q_{i-r+2}\cdots
q_{i+n-2r+1}g^{n-1}_{r-2,i-1}\bar{a}_{i+n-2r+1} \\
& = 0
\end{align*}
as required.
\end{proof}

In order to define $\partial^n$ for $n \pgq 1$ in a minimal
projective bimodule resolution $(P^*, \partial^*)$ of $\lq$, we
use the following notation. In describing the image of
$\mathfrak{o}(g^n_{r,i}) \otimes \mathfrak{t}(g^n_{r,i})$ under
$\partial^n$ in the projective module $P^{n-1}$, we use subscripts
under $\otimes$ to indicate the appropriate summands of the
projective module $P^{n-1}$. Specifically, let $\otimes_r$ denote a
term in the summand of $P^{n-1}$ corresponding to $g^{n-1}_{r,-}$,
and $\otimes_{r-1}$ denote a term in the summand of $P^{n-1}$
corresponding to $g^{n-1}_{r-1,-}$, where the appropriate index $-$
of the vertex may always be uniquely determined from the context.
Indeed, since the relations are uniform along the quiver, we can
also take labelling elements defined by a formula independent of $i$,
and hence we omit the index $i$ when it is clear from the context.
Recall that nonetheless all tensors are over $K$.

The algebra $\lq$ is Koszul, so we now use \cite{GHMS} to give a minimal
projective bimodule resolution $(P^*, \partial^*)$ of $\lq$. We
define the map $\partial^0\colon P^0\to \lq$ to be the
multiplication map. For $n\pgq 1$, we define the map
$\partial^n: P^n\to P^{n-1}$ as follows:
$$\begin{array}{ll}
\partial^n \colon\mathfrak{o}(g^n_{r,i}) \otimes \mathfrak{t}(g^n_{r,i})
\mapsto &
(e_i\otimes_ra_{i+n-2r-1} +
(-1)^n\underbrace{q_{i-r+1}q_{i-r+2}\cdots q_{i+n-2r}}_{\mbox{$n-r$
    terms}} e_i\otimes_{r-1}\bar{a}_{i+n-2r})\\
\phantom{a}\\
& +(-1)^n((-1)^r\underbrace{q_{i-r+1}q_{i-r+2}\cdots q_i}_{\mbox{$r$
    terms}} a_i\otimes_re_{i+n-2r} +
(-1)^r\bar{a}_{i-1}\otimes_{r-1}e_{i+n-2r}).
\end{array}$$
Using our conventions, the degenerate cases $r = 0$ and $r = n$ simplify to
$$\partial^n\colon \mathfrak{o}(g^n_{0,i}) \otimes \mathfrak{t}(g^n_{0,i})
\mapsto e_i\otimes_0a_{i+n-1} + (-1)^na_i\otimes_0e_{i+n}$$ where
the first term is in the summand corresponding to $g^{n-1}_{0,i}$
and the second term is in the summand corresponding to
$g^{n-1}_{0,i+1}$, whilst
$$\partial^n\colon \mathfrak{o}(g^n_{n,i}) \otimes \mathfrak{t}(g^n_{n,i})
\mapsto (-1)^ne_i\otimes_{n-1}\bar{a}_{i-n} +
\bar{a}_{i-1}\otimes_{n-1}e_{i-n},$$ with the first term in the
summand corresponding to $g^{n-1}_{n-1,i}$ and the second term in
the summand corresponding to $g^{n-1}_{n-1,i-1}$.

We claim that the map $\partial^n$ does indeed make $(P^*, \partial^*)$ into a complex.

\bl\label{lem:complex}
We have $\partial^{n}\circ \partial^{n+1} = 0$.
\el

\begin{proof}
The proof is a matter of applying the two different recursive formulae
for $g^{n}_{r,i}$. It is not difficult, but care is needed with all
the terms. We have

\noindent\begin{align*}
&\dd^n \circ \dd^{n+1}
(\mathfrak{o}(g^{n+1}_{r,i}) \otimes
  \mathfrak{t}(g^{n+1}_{r,i}))\\
& =
\dd^n\bigl(
(e_i\otimes_ra_{i+n-2r} +
(-1)^{n+1}q_{i-r+1}q_{i-r+2}\cdots q_{i+n-2r+1}e_i\otimes_{r-1}\bar{a}_{i+n-2r+1})\\
&\quad  +(-1)^{n+1}((-1)^rq_{i-r+1}q_{i-r+2+1}\cdots q_ia_i\otimes_re_{i+n-2r+1} +
(-1)^r\bar{a}_{i-1}\otimes_{r-1}e_{i+n-2r+1}) \bigr)\\
& =
\dd^n\bigl(
(\mathfrak{o}(g^{n}_{r,i})
\otimes_r\mathfrak{t}(g^{n}_{r,i}) a_{i+n-2r} +
(-1)^{n+1}q_{i-r+1}q_{i-r+2}\cdots
q_{i+n-2r+1}\mathfrak{o}(g^{n}_{r-1,i})\otimes_{r-1}
\mathfrak{t}(g^{n}_{r-1,i}) \bar{a}_{i+n-2r+1})\\
&\quad  -(-1)^{n}((-1)^rq_{i-r+1}q_{i-r+2+1}\cdots
q_ia_i\mathfrak{o}(g^{n}_{r,i+1}) \otimes_r\mathfrak{t}(g^{n}_{r,i+1}) +
(-1)^r\bar{a}_{i-1}\mathfrak{o}(g^{n}_{r-1,i-1}) \otimes_{r-1}
\mathfrak{t}(g^{n}_{r-1,i-1}))  \bigr)\\
& =
e_i \otimes_r a_{i+n-2r-1}a_{i+n-2r} +
(-1)^n q_{i-r+1}\cdots q_{i+n-2r} e_i \otimes_{r-1}
\bar{a}_{i+n-2r}a_{i+n-2r}\\
& \quad +(-1)^{n+r}q_{i-r+1}\cdots q_{i} a_i \otimes_{r} a_{i+n-2r}
+(-1)^{n+r} \bar{a}_{i+n-2r} \otimes_{r-1} a_{i+n-2r}\\
&\quad + (-1)^{n+1} q_{i-r+1}\cdots q_{i+n-2r+1} e_i \otimes_{r-1}
a_{i+n-2r+1}\bar{a}_{i+n-2r+1}\\
&\quad + (-1)^{2n+1} q_{i-r+1}\cdots q_{i+n-2r+1} q_{i-r+2}\cdots q_{i+n-2r+2}
e_i \otimes_{r-2} \bar{a}_{i+n-2r+2}\bar{a}_{i+n-2r+1}\\
&\quad + (-1)^{2n+r} q_{i-r+1}\cdots q_{i+n-2r+1} q_{i-r+2}\cdots q_{i}
a_i \otimes_{r-1} \bar{a}_{i+n-2r+1}\\
&\quad + (-1)^{2n+r} q_{i-r+1}\cdots q_{i+n-2r+1} \bar{a}_{i-1} \otimes_{r-2} \bar{a}_{i+n-2r+1}\\
&\quad - (-1)^{n+r} q_{i-r+1}\cdots q_{i} a_{i} \otimes_{r} a_{i+n-2r}
- (-1)^{2n+r} q_{i-r+1}\cdots q_{i}q_{i-r+2}\cdots q_{i+n-2r+1} a_{i}
\otimes_{r-1} \bar{a}_{i+n-2r+1}\\
&\quad - (-1)^{2n+r} q_{i-r+1}\cdots q_{i}q_{i-r+2}\cdots q_{i+n-2r+1}
a_{i} a_{i+1} \otimes_{r} e_{i+n-2r+1}\\
&\quad - (-1)^{2n+2r} q_{i-r+1}\cdots q_{i} a_i \bar{a}_{i} \otimes_{r-1} e_{i+n-2r+1}\\
&\quad - (-1)^{n+r}   \bar{a}_{i-1} \otimes_{r-1} a_{i+n-2r}
 - (-1)^{2n+r} q_{i-r+1}\cdots q_{i+n-2r+1} \bar{a}_{i-1}
 \otimes_{r-2} \bar{a}_{i+n-2r+1}\\
&\quad - (-1)^{2n+2r-1}  q_{i-r+1}\cdots q_{i-1} \bar{a}_{i-1} a_{i-1}
 \otimes_{r} e_{i+n-2r+1}
 - (-1)^{2n+2r} \bar{a}_{i-1}\bar{a}_{i-2} \otimes_{r-1} e_{i+n-2r+1}
\end{align*}

The third term cancels with the 9th term, the 4th with the 13th,
the 7th with the 10th and the 8th with the 14th. We now apply the relations in $\lq$.
Using $a_ia_{i+1} =0 = \bar{a}_i\bar{a}_{i-1}$, we have that the
first, 6th, 11th and 16th terms are zero.
The $q_i a_i \bar{a}_{i} - \bar{a}_{i-1}a_{i-1}$ relations mean that the
2nd and 5th terms cancel, and the 12th and 15th terms cancel.
Thus the net sum is zero, and the result follows.
\end{proof}

The next theorem is now immediate from \cite[Theorem 2.1]{GHMS}.

\bt
With the above notation, $(P^*, \partial^*)$ is a minimal projective bimodule resolution of $\lq$.
\et

\bigskip

\section{The Hochschild cohomology ring of $\lq$}

\bigskip

We consider the complex ${\Hom}_{\lq^e}(P^n, \lq)$.
All our homomorphisms are $\lq^e$-homomorphisms and so we
write ${\Hom}(-,-)$ for ${\Hom}_{\lq^e}(-,-)$. We start by computing the dimension of the space
$\Hom(P^n, \lq)$ for each $n \pgq 0$. For $m \pgq 3$, we write $n=pm+t$ where $p\pgq 0$ and $0\ppq t\ppq
m-1$.

\bl\label{lem:dim} Suppose $m \pgq 3$ and $n=pm+t$ where $p\pgq 0$ and $0\ppq t\ppq
m-1$. Then
$$\dim_K\Hom(P^n, \lq) =
\left\{
\begin{array}{ll}
(4p+2)m & \mbox{if } t\neq m-1 \\
(4p+4)m & \mbox{if } t=m-1.
\end{array}
\right.
$$
If $m = 1$ or $m = 2$ then $$\dim_K\Hom(P^n, \lq) = 4(n+1).$$
\el

The proof is as for the non-deformed case (with $q_0 = q_1 = \cdots = q_{m-1} = 1$)
in \cite[Lemma 1.7]{ST} and where $N = 1$, and so is omitted.

\bigskip

Applying $\Hom(-, \lq)$ to the resolution $(P^*,\partial^*)$ gives the complex
$(\Hom(P^n, \lq), d^n)$ where $d^n : \Hom(P^n, \lq) \to \Hom(P^{n+1}, \lq)$ is induced by the map
$\partial^{n+1} : P^{n+1} \to P^n$. The $n$th Hochschild cohomology group $\HH^n(\lq)$ is then given by
$\HH^n(\lq) = \ker d^n/\im d^{n-1}$. We start by calculating the dimensions of $\ker d^n$ and $\im d^{n-1}$. We consider the cases $m \pgq 3$ and $m=2$ separately, and recall that the Hochschild cohomology of $\lq$ in
the case $m=1$ was fully determined in \cite{BGMS}.

\bigskip

We keep to the notational conventions of \cite{ST}. So far, we have simplified notation by denoting the
idempotent $\mathfrak{o}(g^n_{r,i}) \otimes \mathfrak{t}(g^n_{r,i})$
of the summand $\lq\mathfrak{o}(g^n_{r,i}) \otimes
\mathfrak{t}(g^n_{r,i})\lq$ of $P^n$ uniquely by $e_i\ot_r
e_{i+n-2r}$ where $0 \ppq i \ppq m-1$. However, even this notation
with subscripts under the tensor product symbol becomes cumbersome
in computations. Thus we now recall the additional conventions of \cite[1.3]{ST}
which we keep throughout the rest of the paper. Specifically, since
$e_{i+n-2r} \in \{e_0, e_1, \ldots , e_{m-1}\}$, it would be
usual to reduce the subscript $i+n-2r$ modulo $m$. However, to make
it explicitly clear to which summand of the projective module $P^n$
we are referring and thus to avoid confusion, whenever we write
$e_i\otimes e_{i+k}$ for an element of $P^n$, we will always have $i
\in \{0, 1, \ldots , m-1\}$ and consider $i+k$ as an element of
${\mathbb Z}$, in that $r = (n-k)/2$ and $e_i\otimes e_{i+k} =
e_i\otimes_{\frac{n-k}{2}} e_{i+k}$ and thus lies in the
$\frac{n-k}{2}$-th summand of $P^n$. We do not reduce $i+k$ modulo
$m$ in any of our computations. In this way, when considering
elements in $P^n$, our element $e_i\otimes e_{i+k}$ corresponds
uniquely to the idempotent $\mathfrak{o}(g^n_{r,i}) \otimes
\mathfrak{t}(g^n_{r,i})$ of $P^n$ with $r = (n-k)/2$, for
each $i=0, 1, \ldots , m-1$.

With this notation and for future reference, we note that an element $f \in \Hom(P^n, \lq)$
is determined by its image on each $e_i \otimes e_j$ that generates a summand of
$P^n$. Now $f (e_i \otimes e_j) \in e_i \lq e_j$ and hence can only be
non-zero if $i=j$ or if $i = j \pm 1$. For $m \pgq 3$ and $f \in \Hom(P^{n}, \lq)$ we may write:
$$
\begin{cases}
f(e_i\ot e_{i+\alpha m})=\sigma_{i}^{\alpha}e_i + \tau_{i}^{\alpha}\oa_{i-1}a_{i-1},\\
f(e_i\ot e_{i+\beta m-1})=\lambda_{i}^{\beta}\oa _{i-1},\\
f(e_i\ot e_{i+\gamma m+1})=\mu_{i}^{\gamma}a_i, \end{cases}
$$
with coefficients $\sigma_{i}^{\alpha},$ $\tau_{i}^{\alpha},$
$\lambda_{i}^\beta$ and $\mu_{i}^\gamma$ in $K$, and appropriate ranges of integers
$\alpha$, $\beta$ and $\gamma$. Specifically, for $\lq e_i\otimes e_{i+\alpha m}\lq$ to be a summand of $P^n$,
we require $i+\alpha m = i+n-2r$ for some $0 \ppq r \ppq n$. Similarly we require
$i+\beta m -1 =i+n-2r$ and $i+\gamma m +1 =i+n-2r$ for some $0 \ppq r \ppq n$.
The precise ranges of $\alpha$, $\beta$ and $\gamma$ for the case $m \pgq 3$ are as follows.
(We have four cases based on the parity of $t$ and of $m$, where $n=pm+t$ with $0 \ppq t \ppq m-1$.)

If both $t$ and $m$ are even, then we only need $\alpha$.
We have $2p+1$ values of $\alpha$
with $-p \ppq \alpha \ppq p$.

If $t$ is even and $m$ is odd, then we have
$p+1$ values of $\alpha$
with $-p \ppq \alpha \ppq p$ and $\alpha \equiv p \mod 2$.
For $t \ppq m-2$
we also have
$p$ values of $\beta$ and $\gamma$
with $-p+1 \ppq \beta \ppq p-1$,
$-p+1 \ppq \gamma \ppq p-1$
 and $\beta \equiv \gamma \equiv p +1 \mod 2$.
If $t = m-1$ then we get $p+1$ values of $\beta$ and $\gamma$
with $-p+1 \ppq \beta \ppq p+1$,
$-p-1 \ppq \gamma \ppq p-1$
 and $\beta \equiv \gamma \equiv p +1 \mod 2$.

If $t$ is odd and $m$ is even, then we have
no values for $\alpha$.
For $t \ppq m-2$
we have
$2p+1$ values of $\beta$ and $\gamma$
with $-p \ppq \beta \ppq p$ and
$-p \ppq \gamma \ppq p$.
If $t = m-1$ then we get $2p+2$ values of $\beta$ and $\gamma$
with $-p \ppq \beta \ppq p+1$ and
$-p-1 \ppq \gamma \ppq p$.

If $t$ is odd and $m$ is odd, then we have
$p$ values of $\alpha$
with $-p+1 \ppq \alpha \ppq p-1$ and $\alpha \equiv p+1 \mod 2$.
We also have
$p+1$ values of $\beta$ and $\gamma$
with $-p \ppq \beta \ppq p$,
$-p \ppq \gamma \ppq p$
 and $\beta \equiv \gamma \equiv p  \mod 2$.

\bigskip

We consider the case $m = 2$ in Section \ref{m=2} and suppose now that $m \pgq 3$.

\bigskip

\subsection{$\ker d^n$ where $m \pgq 3$.}

Let $f \in \Hom(P^n, \lq)$ and suppose $f \in \ker d^n$ so that
$d^n(f) = f \circ \partial^{n+1} \in \Hom(P^{n+1}, \lq)$. Write $n = pm + t$ with $0 \ppq t \ppq m-1$.
We evaluate $d^n(f)$ at $e_i\ot e_{i+n+1-2r}$ for $r = 0, \ldots , n+1$.
We have three separate cases for $r$ to consider.

We first consider $r=0$. Then, for each $i = 0, \ldots , m-1$ we have
\begin{align*}
d^n(f)(e_i \ot e_{i+n+1})& =
f\circ \partial^{n+1}(e_i \ot e_{i+n+1})\\
&=
f(e_i\ot e_{i+n}) a_{i+n}  + (-1)^{n+1} a_i f(e_{i+1}\ot e_{i+n+1}) \\
&=
\begin{cases}
\lambda^{p+1}_i\oa_{i-1} a_{i-1} -(-1)^n\lambda ^{p+1}_{i+1}a_i\oa_i
&\mbox{ if $t=m-1$}\\
\mu^{p}_ia_{i} a_{i+1} -(-1)^n\mu ^{p}_{i+1}a_ia_{i+1}
&\mbox{ if $t=1$}\\
\sigma^{p}_ia_{i} +\tau^p_i \oa_{i-1} a_{i-1}a_i -
(-1)^n(\sigma^{p}_{i+1}a_{i} +\tau^p_{i+1} a_i\oa_{i} a_{i})
&\mbox{ if $t=0$}\\
0
&\mbox{ otherwise.}
\end{cases}
\end{align*}
Applying the relations in $\lq$ gives:
\begin{align*}
d^n(f)(e_i \ot e_{i+n+1})&=
\begin{cases}
(q_i\lambda^{p+1}_i -(-1)^n\lambda ^{p+1}_{i+1})a_i\oa_i
&\mbox{ if $t=m-1$}\\
0
&\mbox{ if $t=1$}\\
(\sigma^{p}_i -(-1)^n \sigma^{p}_{i+1})a_{i}
&\mbox{ if $t=0$}\\
0
&\mbox{ otherwise.}
\end{cases}
\end{align*}

Thus if $f \in \ker d^n$ and $t = m-1$ this gives
the condition
\begin{align*}
\lambda ^{p+1}_{i+1}
&= (-1)^n q_i\lambda^{p+1}_i
= (-1)^{2n} q_iq_{i-1} \lambda^{p+1}_{i-1}
=(-1)^{3n} q_iq_{i-1}q_{i-2} \lambda^{p+1}_{i-2}\\
&= \cdots
=(-1)^{mn} q_iq_{i-1}q_{i-2}\cdots q_{i-m+1} \lambda^{p+1}_{i+1}
\end{align*}
and hence
$$\lambda^{p+1}_{i+1} = (-1)^{mn} \zeta \lambda^{p+1}_{i+1}.$$
So to get non-trivial solutions for $\lambda^{p+1}_{i+1}$ we need
$\zeta = (-1)^{mn}$.
But we assumed that $\zeta$ is not a root of unity and thus there are no
non-trivial solutions for $\lambda^{p+1}_{i+1}$, that is,
$\lambda^{p+1}_i =0$ for all $i$.

If $f \in \ker d^n$ and $t=0$ this gives
the condition
$$
\sigma ^{p}_{i+1}
= (-1)^n \sigma^{p}_i
= (-1)^{2n} \sigma^{p}_{i-1}
= \cdots = (-1)^{mn} \sigma^p_{i+1}
$$
and so to get non-trivial solutions for $\sigma^{p}_{i+1}$ we need
$$
(-1)^{mn}  =1.
$$
Now note that each $\sigma_i^p$ is determined by the others, so
we need only determine one of them, say $\sigma_0^p$.
Then we will have a free choice for $\sigma_0^p$ if $mn$ is even or
$\car K =2$, but $\sigma_0^p = 0$ (and hence $\sigma_i^p = 0$ for all $i$) if $mn$ is
odd and $\car K \ne 2$.

So if $r = 0$ then, for $f$ to be in $\ker d^n$, we have the conditions:
$$
\begin{cases}
\lambda_i^{p+1} =0 &\mbox{ if $t = m-1$}\\
\sigma_i^{p} = 0 & \mbox{ if $t = 0$ and  $(-1)^{mn} \ne 1$}\\
\sigma_i^{p} = (-1)^{in}\sigma_0^p & \mbox{ if $t=0$ and  $(-1)^{mn} = 1$}\\
\end{cases}
$$
for all $i = 0, \ldots , m-1$.

\bigskip

We next consider $ r = n +1$. Then
\begin{align*}
d^n(f)(e_i &\ot e_{i-n-1})=
f\circ \partial^{n+1}(e_i \ot e_{i-n-1})\\
&=
(-1)^{n+1}f( e_i\ot e_{i-n}) \oa_{i-n-1}  +  \oa_{i-1} f(e_{i-1}\ot e_{i-n-1}) \\
&=
\begin{cases}
-(-1)^n\mu^{-p-1}_ia_{i} \oa_{i} +\mu ^{-p-1}_{i-1}\oa_{i-1}a_{i-1}
&\mbox{ if $t = m-1$}\\
-(-1)^n\lambda^{-p}_i\oa_{i-1} \oa_{i-2} +
\lambda^{-p}_{i-1}\oa_{i-1}\oa_{i-2}
&\mbox{ if $t=1$}\\
-(-1)^n(\sigma^{-p}_i\oa_{i-1} +\tau^{-p}_i \oa_{i-1} a_{i-1}\oa_{i-1})
+\sigma^{-p}_{i-1}\oa_{i-1} +\tau^{-p}_{i-1} \oa_{i-1}\oa_{i-2} a_{i-2}
&\mbox{ if $t=0$}\\
0
&\mbox{ otherwise.}
\end{cases}
\end{align*}
Applying the relations in $\lq$ gives:
$$
d^n(f)(e_i \ot e_{i-n-1})
=\begin{cases}
(q_i\mu ^{-p-1}_{i-1}-(-1)^n\mu^{-p-1}_i )a_{i}\oa_{i}
&\mbox{ if $t=m-1$}\\
0
&\mbox{ if $t=1$}\\
(\sigma^{-p}_{i-1}-(-1)^n\sigma^{-p}_i )\oa_{i-1}
&\mbox{ if $t=0$}\\
0
&\mbox{ otherwise.}
\end{cases}
$$

If $t=0$ then we will get that the $\sigma^{-p}_i$ are
all dependent on $\sigma^{-p}_0$, and they will be all zero if $mn$ is odd
and $\car K \ne 2$. If $t=m-1$ then all the $\mu^{-p-1}_i$ are zero.

So if $r = n+1$ then, for $f$ to be in $\ker d^n$, we have the conditions:
$$
\begin{cases}
\mu_i^{-p-1} =0 &\mbox{ if $t=m-1$}\\
\sigma_i^{-p} =0 & \mbox{ if $t=0$ and $(-1)^{mn} \ne 1$}\\
\sigma_i^{-p} = (-1)^{in}\sigma_0^{-p} &\mbox{ if $t=0$ and  $(-1)^{mn} = 1$.}\\
\end{cases}
$$
for all $i = 0, \ldots , m-1$.

We now do the generic case for $r$ with $1 \ppq r \ppq n$. We have
\begin{align*}
&d^n(f)(e_i \ot e_{i+n+1-2r})=
f\circ \partial^{n+1}(e_i \ot e_{i+n+1-2r})\\
&=
f(e_i\otimes_re_{i+n-2r})a_{i+n-2r}
-(-1)^nq_{i-r+1}q_{i-r+2}\cdots q_{i+n-2r+1}
f(e_i\otimes_{r-1}e_{i+n-2r+2})\bar{a}_{i+n-2r+1}\\
& \qquad -(-1)^n((-1)^r
q_{i-r+1}q_{i-r+2}\cdots q_ia_i
f(e_{i+1}\ot_r e_{i+n-2r+1}) +
(-1)^r\bar{a}_{i-1}f(e_{i-1}\ot_{r-1} e_{i+n-2r+1}))\\
&=
\begin{cases}
\sigma^{\alpha}_ia_{i}
 -(-1)^{n+r}q_{i-r+1}q_{i-r+2}\cdots q_i
\sigma^{\alpha}_{i+1}a_{i}
&\mbox{ if $n - 2r = \alpha m$}\\
\phantom{a}\\
(q_i\lambda^{\beta}_i
-(-1)^nq_{i-r+1}q_{i-r+2}\cdots q_{i+n-2r+1}
\mu^{\beta}_{i})a_{i}\oa_{i}\\
\qquad -(-1)^{n+r}
( q_{i-r+1}q_{i-r+2}\cdots q_i
\lambda^{\beta}_{i+1} + q_i\mu ^{\beta}_{i-1})a_i\oa_{i}
&\mbox{ if $n - 2r = \beta m -1$}\\
\phantom{a}\\
-(-1)^nq_{i-r+1}q_{i-r+2}\cdots q_{i+n-2r+1}
\sigma^{\alpha}_{i}\oa_{i-1}
 -(-1)^{n+r} \sigma^{\alpha}_{i-1}\oa_{i-1}
&\mbox{ if $n - 2r = \alpha m -2$}\\
\phantom{a}\\
0
&\mbox{ otherwise.}
\end{cases}
\end{align*}

For $n-2r = \alpha m$ we get
a similar situation to the $r=0$ and $t = m - 1$ case.
We write $r = bm +c$ with $b \in \Z$ and $0\ppq c \ppq m-1$.
We need:
\begin{align*}
\sigma^{\alpha}_i
&= (-1)^{n+r} \underbrace{q_{i-r+1}q_{i-r+2}\cdots q_i}_{\mbox{$r$ terms}}
\sigma^{\alpha}_{i+1}
= (-1)^{n+r} \underbrace{q_{i-r+1}q_{i-r+2}\cdots q_{i-r+c}}_{\mbox{$c$
    terms}} \zeta^b \sigma^{\alpha}_{i+1}\\
&= (-1)^{2n+2r} \underbrace{q_{i-r+1}q_{i-r+2}\cdots q_{i-r+c}}_{\mbox{$c$ terms}}
 \underbrace{q_{i-r+c+1}q_{i-r+c+2}\cdots q_{i-r+2c}}_{\mbox{$c$ terms}}
\zeta^{2b}\sigma^{\alpha}_{i+2}\\
& = \cdots
= (-1)^{mn+mr} \zeta^c
\zeta^{mb}\sigma^{\alpha}_{i}
= (-1)^{mn+mr} \zeta^{r}\sigma^{\alpha}_{i}.
\end{align*}
Thus either all $\sigma^\alpha_i$ are zero or $\zeta$ is a root of
unity. Hence (by assumption) $\sigma^\alpha_i = 0$ for all $i$ and all $\alpha$ with $n-2r = \alpha m$.

For $n-2r = \beta m-1$, the condition that $f$ is in $\ker d^n$ yields the $m$ equations
$$
(-1)^{n+r} \underbrace{q_{i-r+1}q_{i-r+2}\cdots q_i}_{\mbox{$r$ terms }} \lambda^{\beta}_{i+1}
=
q_i\lambda^{\beta}_i
-(-1)^n\underbrace{q_{i-r+1}q_{i-r+2}\cdots q_{i+n-2r+1}}_{\mbox{$n-r
    +1$ terms}}\mu^{\beta}_{i}
-(-1)^{n+r}q_i\mu ^{\beta}_{i-1}
$$
in the $2m$ variables $\lambda^{\beta}_i, \mu^{\beta}_i$ where $i = 0, \ldots , m-1$
(with the obvious conventions that $\lambda^{\beta}_0 = \lambda^{\beta}_{m-1}$ etc.).
We may rewrite these equations as
\begin{align*}
\lambda^{\beta}_{i+1} & =
(-1)^{n+r}(\underbrace{q_{i-r+1}q_{i-r+2}\cdots q_{i-1}}_{\mbox{$r-1$ terms}})^{-1}\lambda^{\beta}_i\\
& \qquad +(-1)^{r+1}(\underbrace{q_{i-r+1}q_{i-r+2}\cdots q_i}_{\mbox{$r$ terms}})^{-1}
\underbrace{q_{i-r+1}q_{i-r+2}\cdots q_{i+n-2r+1}}_{\mbox{$n-r+1$ terms}}\mu^{\beta}_{i}\\
& \qquad -(\underbrace{q_{i-r+1}q_{i-r+2}\cdots q_{i-1}}_{\mbox{$r-1$ terms}})^{-1}\mu ^{\beta}_{i-1}
\end{align*}
since the $q_i$ are invertible.
Thus we may write all the $\lambda^\beta_i$ in terms of $\lambda^\beta_0, \mu^{\beta}_0, \ldots , \mu^{\beta}_{m-1}$.
We may then write
$\lambda^\beta_0$ in terms of $\mu^{\beta}_0, \ldots , \mu^{\beta}_{m-1}$, provided that the coefficient of $\lambda^\beta_0$ is non-zero. Specifically, if $r \ne 1$, then the equations give
\begin{align*}
\lambda^{\beta}_0 & = (-1)^{(n+r)m}
(q_{m-r}\cdots q_{m-2})^{-1}(q_{m-r-1}\cdots q_{m-3})^{-1}\cdots (q_{m-r+1}\cdots q_{m-1})^{-1}
\lambda^{\beta}_0 \\
& \qquad + \mbox{ terms in $\mu^{\beta}_0, \ldots , \mu^{\beta}_{m-1}$}
\end{align*}
and hence
$$\lambda^{\beta}_0 = (-1)^{(n+r)m}(\zeta^{-1})^{r-1}\lambda^{\beta}_0 + \mbox{ terms in $\mu^{\beta}_0, \ldots , \mu^{\beta}_{m-1}$}.$$
Since $\zeta$ is not a root of unity, it follows that we may write $\lambda^\beta_0$ in terms of $\mu^{\beta}_0, \ldots , \mu^{\beta}_{m-1}$. On the other hand, suppose $r=1$. Here the original equations reduce to
$$\lambda^{\beta}_{i+1} = (-1)^{n+1}\lambda^{\beta}_i
 + \underbrace{q_{i+1}q_{i+2}\cdots q_{i+n-1}}_{\mbox{$n-1$ terms}}\mu^{\beta}_{i}
 -\mu ^{\beta}_{i-1}.
$$
If $n$ is even and $\car K \ne 2$ then we can again write $\lambda^\beta_0$ in terms of $\mu^{\beta}_0, \ldots , \mu^{\beta}_{m-1}$. However, if $n$ is odd or $\car K = 2$ then adding these equations together gives
$$\sum_{i=0}^{m-1}((q_{i+1}\cdots q_{i+n-1}) - 1) \mu^{\beta}_i = 0$$
so that there is a dependency among the $\mu^{\beta}_i$ but $\lambda^\beta_0$ is a free variable if $n \ne 1$. (If $n = 1$ then both sides are zero so there is no dependency.)

Finally, we consider the case where $n-2r = \alpha m - 2$. Here we have
the condition:
$$(-1)^nq_{i-r+1}q_{i-r+2}\cdots q_{i+n-2r+1}
\sigma^{\alpha}_{i}
= -(-1)^{n+r} \sigma^{\alpha}_{i-1}.$$
This is similar to the $n-2r = \alpha m$ case and we deduce that
all the $\sigma^\alpha_i$ are zero since $\zeta$ is not a root of unity.

Hence, if $1 \ppq r \ppq n$ and $f$ is in $\ker d^n$, we have:
$$
\begin{cases}
\sigma_i^{\alpha} =0 & \mbox{ if $n - 2r = \alpha m$ or if $n - 2r = \alpha m - 2$}\\
\dim\sp\{\lambda^{\beta}_0, \dots , \lambda^{\beta}_{m-1}, \mu^{\beta}_0, \dots, \mu^{\beta}_{m-1}\} = m &
\mbox{ if $n - 2r =  \beta m - 1$ and either $r \ne 1$ or $n \ne 1$}\\
\dim\sp\{\lambda^{\beta}_0, \dots , \lambda^{\beta}_{m-1}, \mu^{\beta}_0, \dots, \mu^{\beta}_{m-1}\} = m+1 & \mbox{ if $n - 2r =  \beta m - 1$, $r = 1$ and $n = 1$.}\\
\end{cases}
$$

\bigskip

We now combine this information to determine $\dim\ker d^n$.

\bprop
For $m \pgq 3$,
$$\dim \ker d^n
=
\begin{cases}
m+1
&\mbox{ if $n = 0$ or $n=1$}\\
(2p+1)m
&\mbox{ if $n \pgq 2$}.
\end{cases}
$$
\eprop

\begin{proof}
We first do the cases $n=0, 1$. If $n=0$ then $r=0, 1$ and $\alpha = 0$.
Moreover $(-1)^{mn} = 1$, so $\sigma_i^0 = \sigma_0^0$ for all $i$. Thus $\dim\ker d^0 = m+1$.
If $n=1$ then we have $r=0, 1, 2$ and so $n-2r = 1, -1, -3$ respectively. The only condition comes from
the $r=1$ case, where we have free variables $\lambda_0^0, \mu_0^0, \dots , \mu_{m-1}^0$. Thus $\dim\ker d^0 = m+1$.

For $n = pm+ t \pgq 2$ we consider the 4 cases depending on the parity of $t$ and of $m$.

Suppose both $t$ and $m$ are even.
Here we need only consider the possible values of $\sigma^\alpha_i$
and $\tau^\alpha_i$ with $-p\ppq \alpha \ppq p$. We have that all
$\sigma^\alpha_i$ are zero. (Note that if $t=0$ so $n = pm$ then the $r=1$
case where $n-2 = pm - 2$ shows that all the $\sigma^p_i$ are zero
and the $r=n$ case where $n-2n = -pm$ shows that all the $\sigma^{-p}_i$ are zero.) Hence the only
contribution to the kernel is from the $\tau^\alpha_i$
and thus $\dim\ker d^n = (2p+1)m$.

Suppose $t$ even and $m$ odd. If $t \neq m-1$,
we have $(p+1)m$ many $\sigma^\alpha_i$,
$(p+1)m$ many $\tau^\alpha_i$,
$pm$ many $\lambda_i^\beta$ and $pm$ many $\mu^\gamma_i$.
All the $\sigma_i^\alpha$ are zero and the $\tau_i^\alpha$ are free
as for the previous case giving a $(p+1)m$ dimensional contribution.
The dependence between the $\lambda_i^\beta$ and $\mu_i^\beta$ gives
another $pm$ dimensional contribution.
Thus $\dim\ker d^n = (2p+1)m$ if $t \ne m-1$. The case $t=m-1$  is similar, and we note that
$\lambda^{p+1}_i$ and $\mu^{-p-1}_i$ are all zero by the $r=0$ and $r=n+1$ cases respectively.
So $\dim\ker d^n = (2p+1)m$ if $t = m-1$.

Suppose $t$ odd and $m$ even. If $t \ppq m-2$ we have
$(2p+1)m$ many $\lambda_i^\beta$ and $(2p+1)m$ many $\mu^\gamma_i$.
Thus the dependence between the $\lambda_i^\beta$ and $\mu_i^\beta$ gives
a $(2p+1)m$ dimensional contribution.
If $t=m-1$ then we have $\lambda_i^{p+1} =0$ and $\mu_i^{-p-1} =0$ from the $r=0$ and $r=n+1$ cases so
we still get $(2p+1)m$ dimensions. Thus $\dim\ker d^n = (2p+1)m$.

Finally we consider the case where $t$ and $m$ are both odd.
We have $pm$ values of $\sigma_i^\alpha$, $pm$ values of $\tau^\alpha_i$,
$(p+1)m$ values of $\lambda^\beta_i$ and $(p+1)m$ values of $\mu^\gamma_i$.
Again, the dependence between the $\lambda_i^\beta$ and $\mu_i^\beta$ gives
a $(p+1)m$ dimensional contribution. The $\sigma_i^\alpha$ are all zero and
the $\tau^\alpha_i$ are free. Hence $\dim\ker d^n = (2p+1)m$.

This completes the proof.
\end{proof}

Using the rank-nullity theorem we now get the dimension of $\im d^{n-1}$.

\bprop For $m \pgq 3$ and $n=pm+t$ we have
$$\dim \im d^{n-1}
=
\begin{cases}
0
&\mbox{ if $n = 0$}\\
m-1
&\mbox{ if $n = 1$ or $n=2$} \\
(2p+1)m
&\mbox{ if $n \pgq 3$.}
\end{cases}
$$
\eprop

\begin{proof}
The cases $n = 0, 1, 2$ are immediate. For $n \pgq 3$, write $n = pm+t$ with $0 \ppq t \ppq m-1$.
If $t \ne 0$ then $\dim_K \Hom(P^{pm+t-1}, \lq) = (4p+2)m$ and $\dim \ker d^{pm+t-1} =
(2p+1)m$. If $t=0$ then $\dim_K \Hom(P^{pm-1}, \lq) = 4pm$ and $\dim \ker d^{pm-1} =
(2p-1)m$. Thus in both cases we have $\dim \im d^{n-1} =  (2p+1)m$.
\end{proof}

We come now to our main results where we determine the Hochschild cohomology ring of the algebra $\lq$ when
$\zeta$ is not a root of unity.

\bt\label{thm:dimHH} For $m \pgq 3$,
$$
\dim \HH^n(\lq)
=
\begin{cases}
m+1 & \mbox{if $n =0 $}\\
2  & \mbox{if $n =1 $}\\
1 & \mbox{if $n =2 $}\\
0 & \mbox{if $n \pgq 3$.}
\end{cases}
$$
Thus $\HH^*(\lq)$ is a finite-dimensional algebra of dimension $m+4$.
\et

\begin{thm}\label{thm:ring} For $m\pgq 3$, we have
$$\HH^*(\lq) \cong K[x_0, x_1, \ldots , x_{m-1}] / (x_ix_j) \times_K
\textstyle{\bigwedge} (u_1,u_2)
$$
where $\times_K$ denotes the fibre product over $K$, $\bigwedge(u_1,u_2)$ is the exterior algebra on the generators $u_1$ and $u_2$, the
$x_i$ are in degree $0$, and the $u_i$ are in degree 1.
\end{thm}

\begin{proof}
Since $\HH^0(\lq)$ is the centre $Z(\lq)$, it is clear that $\HH^0(\lq)$ has $K$-basis $\{1, x_0, \ldots , x_{m-1}\}$
where $x_i = a_i\bar{a}_i$. Thus $\HH^0(\lq) = K[x_0, x_1, \ldots , x_{m-1}] / (x_ix_j)$.

Define bimodule maps $u_1, u_2 : P^1 \to \lq$ by
$$\begin{array}{ll}
u_1 : &
\begin{cases}
\mo(g_{0, i}^1) \otimes \mt(g_{0, i}^1) & \mapsto \ \  a_i \mbox{ for all $i = 0, 1, \ldots , m-1$}\\
\mbox{else } & \mapsto \ \  0,
\end{cases}
\\
\phantom{a} &\\
u_2 : &
\begin{cases}
\mo(g_{0, m-1}^1) \otimes \mt(g_{0, m-1}^1) & \mapsto \ \  a_{m-1}\\
\mo(g_{1, 0}^1) \otimes \mt(g_{1, 0}^1) & \mapsto \ \  \bar{a}_{m-1}\\
\mbox{else } & \mapsto \ \  0.
\end{cases}
\end{array}
$$
It is straightforward to show that these maps are in $\ker d^1$ and that they represent
linearly independent elements in $\HH^1(\lq)$ which we also denote by $u_1$ and $u_2$.
Hence $\{u_1, u_2\}$ is a $K$-basis for $\HH^1(\lq)$.

In order to show that $u_1u_2$ represents a non-zero element of $\HH^2(\lq)$, we
define bimodule maps $\lift^0(u_2): P^1 \to P^0$ and $\lift^1(u_2): P^2 \to P^1$ by
$$\begin{array}{ll}
\lift^0(u_2) : &
\begin{cases}
\mo(g_{0, m-1}^1) \otimes \mt(g_{0, m-1}^1) & \mapsto \ \  a_{m-1}\otimes e_0\\
\mo(g_{1, 0}^1) \otimes \mt(g_{1, 0}^1) & \mapsto \ \ \bar{a}_{m-1}\otimes e_{m-1}\\
\mbox{else } & \mapsto \ \  0,
\end{cases}\\
\phantom{a} &\\
\lift^1(u_2) : &
\begin{cases}
\mo(g_{0, m-1}^2) \otimes \mt(g_{0, m-1}^2) & \mapsto \ \  a_{m-1}\mo(g_{0, 0}^1) \otimes \mt(g_{0, 0}^1)\\
\mo(g_{1, 0}^2) \otimes \mt(g_{1, 0}^2) & \mapsto \ \ \bar{a}_{m-1}\mo(g_{0, m-1}^1) \otimes \mt(g_{0, m-1}^1)\\
\mo(g_{1, m-1}^2) \otimes \mt(g_{1, m-1}^2) & \mapsto \ \  -q_{m-1}a_{m-1}\mo(g_{1, 0}^1) \otimes \mt(g_{1, 0}^1)\\
\mo(g_{2, 0}^2) \otimes \mt(g_{2, 0}^2) & \mapsto \ \  -\bar{a}_{m-1}\mo(g_{1, m-1}^1) \otimes \mt(g_{1, m-1}^1)\\
\mbox{else } & \mapsto \ \  0.
\end{cases}
\end{array}$$
Then the following diagram is commutative
$$\xymatrix{
P^2\ar[rr]^{\partial^2}\ar[d]_{\lift^1(u_2)} & & P^1\ar[d]_{\lift^0(u_2)}\ar[dr]^{u_2} & \\
P^1\ar[rr]^{\partial^1} & & P^0\ar[r] & \lq
}$$
where $P^0 \to \lq$ is the multiplication map.
Thus the element $u_1u_2 \in \HH^2(\lq)$ is represented by the map $u_1\circ\lift^1(u_2) : P^2 \to \lq$, that is, by the map
$$\begin{cases}
\mo(g_{1, 0}^2) \otimes \mt(g_{1, 0}^2) & \mapsto \ \  \bar{a}_{m-1}a_{m-1}\\
\mbox{else } & \mapsto \ \  0.
\end{cases}$$
Since this map is not in $\im d^1$, it follows that $u_1u_2$ is non-zero in $\HH^2(\lq)$ and hence $\HH^2(\lq) = \sp\{u_1u_2\}$.

From the lifting $\lift^1(u_2)$ it is easy to see that $u_2^2$ represents the zero element in $\HH^2(\lq)$, and a similar calculation shows that $u_1^2$ also represents the zero element in $\HH^2(\lq)$.
(Note that although it is immediate from the graded commutativity of $\HH^*(\lq)$ that
$u_1^2 = 0 = u_2^2$ in $\HH^2(\lq)$ when $\car K \neq 2$, this direct calculation is required when
$\car K = 2$.)

Thus we have elements $u_1$ and $u_2$ in $\HH^1(\lq)$ which are annihilated by all the $x_i \in \HH^0(\lq)$ and with
$u_1^2 = 0 = u_2^2$ and $u_1u_2 = - u_2u_1$ (with the latter by the graded-commutativity of $\HH^*(\lq)$). Thus
$$\HH^*(\lq) \cong K[x_0, x_1, \ldots , x_{m-1}] / (x_ix_j) \times_K
\bigwedge(u_1,u_2)$$
where $\times_K$ denotes the fibre product over $K$, $\bigwedge(u_1,u_2)$ is the exterior algebra
on the generators $u_1$ and $u_2$, the
$x_i$ are in degree $0$, and the $u_i$ are in degree 1.
\end{proof}

\bigskip

\section{The case $m=2$.}\label{m=2}

\bigskip

We assume that $m=2$ throughout this section.
Recall from Lemma \ref{lem:dim} that $\dim_K\Hom(P^n, \lq) = 4(n+1).$
For $f \in \Hom(P^{n}, \lq)$ we may write:
$$
\begin{cases}
f(e_i\ot e_{i+2\alpha })=\sigma_{i}^{\alpha}e_i +
\tau_{i}^{\alpha}\oa_{i-1}a_{i-1} &\mbox{ if $n$ even}\\
f(e_i\ot e_{i+2\beta +1})=\lambda_{i}^{\beta}\oa _{i-1} +
\mu_{i}^{\beta}a_i & \mbox{ if $n$ odd,}\end{cases}
$$
with coefficients $\sigma_{i}^{\alpha},$ $\tau_{i}^{\alpha},$
$\lambda_{i}^\beta$ and $\mu_{i}^\beta$ in $K$.
The choices of $\alpha$ and $\beta$ are:
$$
\begin{cases}
-p \le \alpha \le p  &\mbox{ if $n$ is even}\\
-p-1 \le \beta \le p  &\mbox{ if $n$ is odd,}
\end{cases}
$$
which gives $n+1$ values in each case.

\bigskip

\subsection{$\ker d^n$ where $m = 2$.}

Let $f \in \Hom(P^n, \lq)$ and suppose $f \in \ker d^n$ so that
$d^n(f) = f \circ \partial^{n+1} \in \Hom(P^{n+1}, \lq)$. Write $n =
2p + t$ with $t = 0, 1$.
We evaluate $d^n(f)$ at $e_i\ot e_{i+n+1-2r}$ for $r = 0, \ldots , n+1$.
There are three cases to consider.

We first consider $r=0$. Then, for $i = 0, 1$
and after applying the relations in $\lq$ we have:
\begin{align*}
d^n(f)(e_i \ot e_{i+n+1})
&=
\begin{cases}
(q_i\lambda^{p}_i + \lambda^{p}_{i+1})a_i\oa_i
&\mbox{ if $t=1$}\\
(\sigma^{p}_i -\sigma^{p}_{i+1})a_{i}
&\mbox{ if $t=0$.}
\end{cases}
\end{align*}
Thus if $f \in \ker d^n$ and $t = 1$ this gives
the condition
\begin{align*}
\lambda ^{p}_{i+1}
= - q_i\lambda^{p}_i
=  q_iq_{i-1} \lambda^{p}_{i+1}
=  \zeta \lambda^{p}_{i+1}
\end{align*}
and hence $\lambda^p_{1}= \lambda^p_0 = 0$ as $\zeta \ne 1$.
If $t=0$ this gives the condition
$\sigma ^{p}_{i+1}
=  \sigma^{p}_i$
and so there is a one-parameter solution here.

We next consider $r = n +1$. After applying the relations in $\lq$ we have:
$$
d^n(f)(e_i \ot e_{i-n-1})=
\begin{cases}
(q_i\mu ^{-p-1}_{i-1}+\mu^{-p-1}_i )a_{i}\oa_{i}
&\mbox{ if $t=1$}\\
(\sigma^{-p}_{i-1}-\sigma^{-p}_i )\oa_{i-1}
&\mbox{ if $t=0$.}
\end{cases}
$$
If $t=0$ then we get a one-parameter solution for the
$\sigma^{-p}_i$ as before. If $t=1$ then we have $\mu^{-p-1}_0 = \mu^{-p-1}_{1} =0$.

We now do the generic case for $r$ with $1 \ppq r \ppq n$. We have $d^n(f)(e_i \ot e_{i+n+1-2r})=$
$$\begin{cases}
(q_i \lambda^{p-2\epsilon}_i + \zeta^{p-\epsilon +1}\mu^{p-2\epsilon+1}_{i}
+\zeta^\epsilon \lambda^{p-2\epsilon}_{i+1} +
q_i\mu^{p-2\epsilon+1}_{i+1} )a_i\oa_{i} &
\mbox{ if $n$ odd and $r=2\epsilon$}\\
q_i( \lambda^{p-2\epsilon-1}_i + \zeta^{p-\epsilon }\mu^{p-2\epsilon}_{i}
-\zeta^\epsilon \lambda^{p-2\epsilon-1}_{i+1} -
\mu^{p-2\epsilon}_{i+1} )a_i\oa_{i} &
\mbox{ if $n$ odd and $r=2\epsilon+1$}\\
(\sigma^{p-2\epsilon}_i - \zeta^{\epsilon} \sigma^{p-2\epsilon}_{i+1})a_{i}
-(q_{i+1}\zeta^{p-\epsilon}\sigma^{p-2\epsilon+1}_i -
\sigma^{p-2\epsilon+1}_{i+1})\oa_{i+1} &
\mbox{ if $n$ even and $r = 2\epsilon$}\\
(\sigma^{p-2\epsilon-1}_i - q_i\zeta^{\epsilon} \sigma^{p-2\epsilon-1}_{i+1})a_{i}
+(-\zeta^{p-\epsilon}\sigma^{p-2\epsilon}_i +
\sigma^{p-2\epsilon}_{i+1})\oa_{i+1} &
\mbox{ if $n$ even and $r = 2\epsilon+1$.}
\end{cases}
$$

If $n$ is even, then
all the $\sigma^\alpha_i$'s are zero and there is no condition
on the $\tau^\alpha_i$'s, as for the $m\pgq 3$ case.

Now suppose $n$ is odd.
If we fix $n$ odd and $r$ even with $1 \ppq r \ppq n$ (so
that $n \pgq 3$) we get a pair of equations:
\begin{align*}
q_i \lambda^{p-2\epsilon}_i + \zeta^{p-\epsilon +1}\mu^{p-2\epsilon+1}_{i}
+\zeta^\epsilon \lambda^{p-2\epsilon}_{i+1} +
q_i\mu^{p-2\epsilon+1}_{i+1}  &= 0\\
\zeta^\epsilon \lambda^{p-2\epsilon}_{i} + q_{i+1}\mu^{p-2\epsilon+1}_{i}
+q_{i+1} \lambda^{p-2\epsilon}_{i+1} + \zeta^{p-\epsilon -1}\mu^{p-2\epsilon+1}_{i+1}
&= 0.
\end{align*}
These equations have a two-parameter solution if and only if
$q_iq_{i+1} -\zeta^{p+1}=
\zeta(1 -\zeta^{p} )\ne 0.$
Since $p \pgq 1$, this is non-zero, and we have a two-parameter
solution.

If we fix $n$ odd and $r$ odd with $1 \ppq r\ppq n$
we get a pair of equations:
\begin{align*}
 \lambda^{p-2\epsilon-1}_i + \zeta^{p-\epsilon }\mu^{p-2\epsilon}_{i}
-\zeta^\epsilon \lambda^{p-2\epsilon-1}_{i+1} -
\mu^{p-2\epsilon}_{i+1} &=0\\
-\zeta^\epsilon \lambda^{p-2\epsilon-1}_{i} -
\mu^{p-2\epsilon}_{i}
 +\lambda^{p-2\epsilon-1}_{i+1} + \zeta^{p-\epsilon }\mu^{p-2\epsilon}_{{i+1}}
&=0.
\end{align*}
These equations have a two-parameter solution if and only if
$-1 +\zeta^{p}\ne 0.$
For $n \pgq 3$ we have $p \pgq 1$, so this is non-zero, and we have a two-parameter
solution.
If however $n=1$ then necessarily $r=n=1$, and we get
\begin{align*}
 \lambda^{-1}_i + \mu^{0}_{i}
- \lambda^{-1}_{i+1} -
\mu^{0}_{i+1} &=0
\end{align*}
which has a three-parameter solution.


\bigskip

We may now determine the dimension of
$\ker d^n$ for $m=2$.

\bprop
For $m = 2$ and $n=2p+t$ with $t=0,1$, we have
$$\dim \ker d^n
=
\begin{cases}
3
&\mbox{ if $n = 0$ or $n=1$}\\
2(2p+1)
&\mbox{ if $n \pgq 2$,}
\end{cases}
$$
and
$$\dim \im d^n
=
\begin{cases}
1
&\mbox{ if $n = 0$}\\
5
&\mbox{ if $n = 1$}\\
2(2p+3)
&\mbox{ if $n \pgq 2$ and $n$ odd}\\
2(2p+1)
&\mbox{ if $n \pgq 2$ and $n$ even}.
\end{cases}
$$
\eprop

\begin{proof}
We first do the cases with small values of $n$. If $n=0$
then $r=0, 1$ and $\alpha =0$.
Thus we get $\dim\ker d^0 = 3$, corresponding to the free variables
$\sigma^0_0$, $\tau^0_0$ and $\tau^0_1$, and hence the image is one-dimensional.
If $n=1$, we have $r=0, 1, 2$ and $n-2r = 1, -1, -3$ respectively.
We need to determine $\mu^0_i$, $\mu^{-1}_i$, $\lambda_i^0$ and
$\lambda_i^{-1}$ ($8$ variables in total).
The $r=0$ case gives
$\lambda^0_i =0$, and the $r=2$ case gives $\mu^{-1}_i=0$.
The $n=r=1$ case has a three-parameter solution for $\lambda_i^{-1}$ in terms of
$\mu^{0}_{i}$,  $ \mu^{0}_{{i+1}}$ and
$\lambda^{-1}_{i+1}$.
Thus overall we get $\dim\ker d^1 = 3$ and $\dim\im d^1 = 5$.

We next consider the case $n=2p$; we need only consider the possible values of $\sigma^\alpha_i$
and $\tau^\alpha_i$ with $-p\ppq \alpha \ppq p$.
Here we get that all $\sigma^\alpha_i$ are zero. So the only
contribution to the kernel is from the $\tau^\alpha_i $
and thus the kernel has dimension $2(2p+1)$.
The dimension of the image
is thus $4(2p+1) - 2(2p+1) = 2(2p+1)$.

Finally, we consider $n=2p +1$.
Here we have
$2(2p+2)$ many $\lambda_i^\beta$ and $2(2p+2)$ many $\mu^\beta_i$
with $-p-1 \ppq \beta \ppq p$.
All the $\lambda_i^\beta$ are dependent on the $\mu^{\beta+1}_i$ if
$-p-1\ppq \beta \ppq p-1$, so we
get a $2(2p+1)$ dimensional contribution here.
Moreover $\lambda^p_0=0=\lambda^p_1$ and
$\mu_0^{-p-1} =0=\mu^{-p-1}_1$. Hence $\dim\ker d^n = 2(2p+1)$.
This gives $\dim\im d^n = 4(2p+2) - 2(2p+1) = 2(2p+3)$.
\end{proof}

Noting that $\dim\HH^0(\lq) = 3 = m+1$, we combine these results with Theorem \ref{thm:dimHH} to give the following theorem.

\bt For $m \pgq 2$,
$$
\dim \HH^n(\lq)
=
\begin{cases}
m+1 & \mbox{if $n =0 $}\\
2  & \mbox{if $n =1 $}\\
1 & \mbox{if $n =2 $}\\
0 & \mbox{if $n \pgq 3$.}
\end{cases}
$$
Thus $\HH^*(\lq)$ is a finite-dimensional algebra of dimension m+4.
\et

It can be verified directly that the proof of Theorem \ref{thm:ring} also holds when $m=2$. Hence we have the following result which describes the ring structure of $\HH^*(\lq)$ when $m=2$ and $\zeta$ is not a root of unity.

\begin{thm} For $m=2$, we have
$$\HH^*(\lq) \cong K[x_0, x_1] / (x_ix_j) \times_K
\textstyle{\bigwedge} (u_1,u_2)
$$
where $\times_K$ denotes the fibre product over $K$, $\bigwedge(u_1,u_2)$ is the exterior algebra on the generators $u_1$ and $u_2$, and the elements $x_0, x_1$ are in degree $0$ and $u_1, u_2$ in degree 1.
\end{thm}

We end by remarking that we have exhibited self-injective algebras whose Hochschild cohomology ring is of arbitrarily large, but nevertheless finite, dimension. The case $m=1$ was studied in \cite{BGMS} where it was shown that the Hochschild cohomology ring is 5-dimensional when $\zeta$ is not a root of unity. Thus, for all $m \pgq 1$, we now have self-injective algebras whose Hochschild cohomology ring is $(m+4)$-dimensional. Hence, for each $N \pgq 5$ we have an algebra with $N-4$ simple modules, of dimension $4(N-4)$ and with infinite global dimension whose Hochschild cohomology ring is $N$-dimensional.

\end{document}